\documentclass[a4paper,oneside,reqno,english]{amsart}
\usepackage[latin9]{inputenc}
\usepackage{tikz}
\usepackage{caption}
\usepackage{amsmath}
\usepackage{bbm}
\usepackage{bm}
\usepackage{babel}
\usepackage{verbatim}
\usepackage{mathtools}
\usepackage{amsthm}
\usepackage{amstext}
\usepackage{amssymb}
\usepackage{amsfonts}
\usepackage{comment}
\usepackage{mathabx}
\usepackage{esint}
\usepackage[unicode=true,
bookmarks=false,
breaklinks=false,pdfborder={0 0 1},backref=false,colorlinks=false]
{hyperref}
\hypersetup{
	colorlinks,linkcolor=blue,anchorcolor=blue,citecolor=blue}
\usepackage{breakurl}
\usepackage{mathrsfs}
\usepackage{cases}
\makeatletter

\numberwithin{equation}{section}
\numberwithin{figure}{section}
\usepackage{enumitem}		

\theoremstyle{plain}
\newtheorem{thm}{Theorem}[section]
\theoremstyle{plain}

\theoremstyle{remark}
\newtheorem{rem}[thm]{Remark}

\theoremstyle{plain}

\theoremstyle{plain}

\theoremstyle{plain}
\newtheorem{lem}[thm]{Lemma}
\theoremstyle{definition}

\theoremstyle{definition}

\theoremstyle{definition}

\newcommand{\real}{\mathbb{R}}

\newcommand{\A}{\mathcal{A}}

\newcommand{\E}{\mathcal{E}}

\newcommand{\sgn}{\text{sgn}}
\newcommand{\tr}{\textnormal{Tr}}

\begin{document}
\title{Optimal 2-uniform convexity of Schatten classes revisited}
\author{Haonan Zhang}
\maketitle

\begin{abstract}

The optimal 2-uniform convexity of Schatten classes $S_p, 1<p\le 2$ was first proved by Ball, Carlen and Lieb \cite{BCL94}. In this note we revisit this result using multiple operator integrals and generalized monotone metrics in quantum information theory.
\end{abstract}

\section{Introduction}
A normed space $\left(X,\|\cdot\|\right)$ is said to be $r$-uniformly convex for some $r\in [2,\infty)$ if
\begin{equation}\label{ineq:defn of uniform convexity}
\|a+b\|^r+\|a-b\|^r\ge 2\|a\|^r+2\|K^{-1}b\|^r
\end{equation}
for some $K>0$ and for all $a,b\in X$, and is said to be $s$-uniformly smooth for some $s\in (1,2]$ if 
\begin{equation}\label{ineq:defn of uniform smoothness}
\|a+b\|^s+\|a-b\|^s\le 2\|a\|^s+2\|Kb\|^s
\end{equation}
for some $K>0$ and for all $a,b\in X$. The best constant $K$ in \eqref{ineq:defn of uniform convexity} for $X$ coincides with the best constant $K$ in \eqref{ineq:defn of uniform smoothness} for $X^*$, the dual of $X$, when $1/r+1/s=1$. See \cite{BCL94} for details.

The classical $L_p$-spaces, $1< p\le 2$, are 2-uniformly convex with optimal constant $(p-1)^{-1/2}$:
\begin{equation*}
\| f+g\|_p^2+\| f-g \|_p^2 \ge 2\|f\|_p^2+2(p-1)\|g\|_p^2,
\end{equation*}
for all $f,g\in L_p$. The inequality is reversed when $2\le p<\infty$. See \cite{BCL94} for two proofs. The analogs of these optimal 2-uniform convexity inequalities for Schatten classes $S_p,1<p\le 2$ were proved by Ball, Carlen and Lieb \cite{BCL94}, which have plenty of applications, for example the optimal hypercontractivity for fermi fields \cite{CL93}. Namely they proved that \cite[Proof of Theorem 1]{BCL94}
\begin{thm}\label{thm:matrix}
	For $1< p\le 2$ and any matrices $A,B$, we have
	\begin{equation}\label{ineq:2-uniform convexity matrix}
	\| A+B\|_p^2+\| A-B \|_p^2 \ge 2\|A\|_p^2+2(p-1)\|B\|_p^2.
	\end{equation}
	The inequality is reversed for $2\le p<\infty$.
\end{thm}

Following their proof, Ricard and Xu \cite{RX16} extended this result to  noncommutative $L_p$-spaces associated with general von Neumann algebras, and applied it to hypercontractivity for free group von Neumann algebras. In this paper we revisit the proof of Theorem \ref{thm:matrix} using multiple operator integrals and generalized monotone metrics in quantum information theory, inspired by two recent preprints \cite{CHPPR20} and \cite{Li20}.

\section{Some lemmas}
\label{sect:2 lemmas}
We start with two lemmas on divided difference of order $n=1,2$. For an open set $\Omega\subset \real$, denote by $C^2(\Omega)$ the collection of all twice continuously differentiable functions on $\Omega$. For $f\in C^2(\Omega)$ and $r,s,t\in \Omega$, put
\begin{equation}
f^{[1]}(r,s):=
\begin{cases*}
\frac{f(r)-f(s)}{r-s} & $r\neq s$\\
f'(r)& $r=s$
\end{cases*}
\end{equation}
and 
\begin{equation}
f^{[2]}(r,s,t):=
\begin{cases*}
\frac{f^{[1]}(r,s)-f^{[1]}(r,t)}{s-t} & $s\neq t$\\
\frac{\partial}{\partial \lambda}|_{\lambda=s}f^{[1]}(r,\lambda)&$s=t$
\end{cases*}.
\end{equation}

The following lemma is stated for $f\in C^2(\real)$, but it is valid for all $f\in C^2(\Omega)$ where $\Omega\subset \real$ is an open set containing the spectrum of $A+tB$ for $|t|$ small enough.
\begin{lem}\cite[Theorem 5.3.2]{ST19}\label{lem:double operators}
	Suppose that $A$ and $B$ are two self-adjoint matrices.
 Then for any $f\in C^2(\real)$ we have
	\begin{equation*}
	\left.\dfrac{d^2}{dt^2}\right|_{t=0}f(A+tB)
	=2\sum_{i=1}^{n}\sum_{j=1}^{n}\sum_{k=1}^{n}f^{[2]}(\lambda_i,\lambda_j,\lambda_k) E^A_i B E^A_j B E^A_k,
	\end{equation*}
	where $A=\sum_{i=1}^{n}\lambda_i E^A_i$ is the spectral decomposition of $A$.
\end{lem}

The following lemma collects a few properties of $f_p^{[2]}$ for $f_p(x)=|x|^p$.
\begin{lem}\label{lem:properties of fp}
	For $1< p\le 2$, put $f_p(x):=|x|^p$. Then for any $r,s\ne 0$ we have
	\begin{enumerate}
		\item $f_p^{[2]}(r,r,r)= f_p^{[2]}(|r|,|r|,|r|)$;
		\item $f_p^{[2]}(r,s,r)+f_p^{[2]}(s,r,s)\ge f_p^{[2]}(|r|,|s|,|r|)+f_p^{[2]}(|s|,|r|,|s|)$ when $r\ne s$.
	\end{enumerate}
\end{lem}

\begin{proof}
	(1) is a direct computation: when $r\ne 0$,
	\begin{equation}\label{eq:r r r}
	f_p^{[2]}(r,r,r)=\frac{1}{2}f_p''(r)=\frac{1}{2}f_p''(|r|)=\frac{1}{2}p(p-1)|r|^{p-2}.
	\end{equation}
	(2) By definition, when $r\ne s$, we have
	\begin{equation*}
	f_p^{[2]}(r,s,r)=\frac{f_p^{[1]}(r,s)-f_p'(r)}{s-r}=\frac{f_p(s)-f_p(r)-(s-r)f_p'(r)}{(s-r)^2}.
	\end{equation*}
	Then 
	\begin{align*}
	&f_p^{[2]}(r,s,r)+f_p^{[2]}(s,r,s)\\
	=&\frac{f_p(s)-f_p(r)-(s-r)f_p'(r)}{(s-r)^2}+\frac{f_p(r)-f_p(s)-(r-s)f_p'(s)}{(r-s)^2}\\
	=&\frac{f_p'(r)-f_p'(s)}{r-s}.
	\end{align*}
	Since $f_p'(r)=p|r|^{p-1}\sgn(r)$, we get 
	\begin{equation}\label{eq:r,s}
	f_p^{[2]}(r,s,r)+f_p^{[2]}(s,r,s)
	=\begin{cases*}
	\frac{p\left(|r|^{p-1}-|s|^{p-1}\right)}{|r|-|s|}& $rs>0$\\
	\frac{p\left(|r|^{p-1}+|s|^{p-1}\right)}{|r|+|s|}& $rs<0$
	\end{cases*}.
	\end{equation}
	The proof of (2) will be finished if
	\begin{equation*}
	\frac{|r|^{p-1}-|s|^{p-1}}{|r|-|s|}\le \frac{|r|^{p-1}+|s|^{p-1}}{|r|+|s|},
	\end{equation*}
	or equivalently
	\begin{equation*}
	\frac{2|rs|^{p-1}\left(|r|^{2-p}-|s|^{2-p}\right)}{\left(|r|+|s|\right)\left(|r|-|s|\right)}\ge 0.
	\end{equation*}
	This is trivial, since $p\le 2$ and thus $x\mapsto x^{2-p}$ is non-decreasing on $(0,\infty)$.
\end{proof}

In \cite{CHPPR20} the authors used 
$$f_p^{[2]}(r,s,r)+f_p^{[2]}(s,r,s)\ge 0,~~r,s\ne 0$$
to prove 
\begin{equation*}
\left.\dfrac{d^2}{dt^2}\right|_{t=0}\tr \left(f_p(A+tB)\right)\ge 0,
\end{equation*}
for self-adjoint $A,B$. The key of this paper can be understood as an improvement of this result, using the following lemma on the joint convexity/the monotonicity property of certain trace functionals. Either of the two properties is enough for our use. The former (Lemma \ref{lem:monotonicity} (1)) is a combination of the results of Hiai-Petz \cite{HP12quasi} and Besenyi-Petz \cite{BP11}. The latter (Lemma \ref{lem:monotonicity} (2)) is essentially from a recent preprint of Li \cite{Li20} (see Proposition 2.5, Proposition 2.7 and Remark 2.8 therein), where she used double operator integrals to generalize the monotone metrics in quantum information theory. This monotonicity result has its own interest, but the details of its proof were not fully provided in \cite{Li20}. For the reader's convenience, we give a detailed proof in Appendix \ref{appendix}.


We need a few notations. For any function $F:(0,\infty)\times (0,\infty)\to (0,\infty)$, and any positive definite matrices $A,B$ with spectral decompositions $A=\sum_{j}\lambda_j E^A_j$ and $B=\sum_{k}\mu_k E^B_k$, put
\begin{equation}
Q_{F}^{A,B}(X):=\sum_{j,k}F(\lambda_j,\mu_k) E^A_j X E^B_k.
\end{equation}
We use $\langle A,B\rangle:=\tr (A^*B)$ to denote the Hilbert-Schmidt inner product. 

\begin{lem}\label{lem:monotonicity}
	For $0< \alpha<1$, put $f_{\alpha}(x):=x^\alpha,x>0$. Let $A,B$ be any positive definite matrices and $X$ be any matrix.
	\begin{enumerate}
		\item The trace functionals
		\begin{equation*}
		(A,B,X)\mapsto \left\langle X, Q_{f_{\alpha}^{[1]}}^{A,B}(X)\right\rangle
		\end{equation*}
		are jointly convex.
		\item For any unital completely positive trace preserving map $\beta$, we have
		\begin{equation}\label{ineq:monotonicity}
		\left\langle \beta(X), Q_{f_{\alpha}^{[1]}}^{\beta(A),\beta(B)}(\beta(X))\right\rangle 
		\le \left\langle X, Q_{f_{\alpha}^{[1]}}^{A,B}(X)\right\rangle.
		\end{equation}
	\end{enumerate}
\end{lem}

\begin{proof}
	(1) Denote by $L_A$ and $R_B$ the left multiplication and right multiplication operators respectively, i.e. $L_A(C):=AC$ and $R_B(C):=CB$. Hiai and Petz \cite[Theorem 7]{HP12quasi} proved that if $\theta\in (0,1]$ and $f:(0,\infty)\to(0,\infty)$ is operator monotone, then 
	\begin{equation}\label{eq:I f theta}
	(A,B,X)\mapsto I^\theta_f(A,B,X):=\left\langle X, (f(L_A R_B^{-1})R_B)^{-\theta}(X)\right\rangle,
	\end{equation}
	is jointly convex, where $A$ and $B$ are positive definite matrices and $X$ is any matrix. If $A$ and $B$ have the spectral decompositions $A=\sum_{j}\lambda_j E^A_j$ and $B=\sum_{k}\mu_k E^B_k$, then 
	\begin{equation*}
	I^\theta_f(A,B,X)=\sum_{j,k} \left[f(\lambda_j/\mu_k)\mu_k\right]^{-\theta}\tr \left(X^*E^A_j XE^B_k\right).
	\end{equation*}
	Note that 
	\begin{equation*}
	f_{\alpha}^{[1]}(r,s)=
	\begin{cases*}
	\frac{r^{\alpha}-s^{\alpha}}{r-s}=\frac{\left(r/s\right)^{\alpha}-1}{r/s-1}\cdot s^{\alpha-1}
	=\left[g_{\alpha}\left(r/s\right)s\right]^{\alpha-1}&$r\ne s$\\
	\alpha s^{\alpha-1} =[g_{\alpha}(1)s]^{\alpha-1}&$r=s$.
	\end{cases*}
	\end{equation*}
	where $g_{\alpha}:(0,\infty)\to (0,\infty)$ is given by
	\begin{equation}\label{eq:g alpha operator monotone}
	g_{\alpha}(x)=
	\begin{cases*}
	\left(\frac{x^{\alpha}-1}{x-1}\right)^{\frac{1}{\alpha-1}}&$x\ne 1$\\
	\alpha^{\frac{1}{\alpha-1}}&$x=1$
	\end{cases*}.
	\end{equation}
	Therefore we have
	\begin{equation*}
	\left\langle X,Q^{A,B}_{f^{[1]}_{\alpha}}(X)\right\rangle
	=I_{g_{\alpha}}^{\theta}(A,B,X),
	\end{equation*}
	with $\theta=1-\alpha\in (0,1)$. It remains to prove that the function $g_{\alpha}$ in \eqref{eq:g alpha operator monotone} is operator monotone when $\alpha\in(0,1)$. This is a result of Besenyi-Petz \cite[Theorem 3]{BP11}.
	
	(2) See Appendix \ref{appendix}.
\end{proof}

\section{Proof of main result}
Now we are ready to reprove Theorem \ref{thm:matrix}.

\begin{proof}[Proof of Theorem \ref{thm:matrix}]
	As we mentioned in the beginning of the introduction, it suffices to prove $1<p\le 2$; see \cite{BCL94} for details. Note that the case $p=2$ is trivial. So let us fix $1<p<2$. By a standard argument we may assume that $A$ and $B$ are self-adjoint. In fact, 
	\begin{equation*}
	A':=\begin{pmatrix}
	0& A\\
	A^*&0
	\end{pmatrix},~~
	B':=\begin{pmatrix}
	0&B\\
	B^*&0
	\end{pmatrix},
	\end{equation*}
	are self-adjoint, and \eqref{ineq:2-uniform convexity matrix}  for $(A',B')$ will yield \eqref{ineq:2-uniform convexity matrix}  for $(A,B)$. By approximation, we may assume that $A$ is invertible. For any $0<\varepsilon<\|A^{-1}B\|^{-1}$, we have
	 \begin{align*}
	 \|(A+tB)^{-1}\|
	 \le \|A^{-1}\|\|(I+tA^{-1}B)^{-1}\|
	 \le \|A^{-1}\|\left(1-|t|\|A^{-1}B\|\right)^{-1}
	 \le M_{\varepsilon}
	 \end{align*}
	 whenever $|t|<\varepsilon$, where $I$ is the identity matrix and $M_{\varepsilon}:=\|A^{-1}\|\left(1-\varepsilon\|A^{-1}B\|\right)^{-1}<\infty$. Therefore for any $|t|<\varepsilon$, the spectrum of $A+tB$ is contained in $\Omega_{\varepsilon}:=\real\setminus [-\frac{1}{2}M_{\varepsilon}^{-1},\frac{1}{2}M^{-1}_{\varepsilon}]$. Then $f_p(x)=|x|^p\in C^2(\Omega_{\varepsilon})$ and we may use Lemma \ref{lem:double operators}.
	 
	By \cite[Proof of Theorem 1]{BCL94}, the proof of \eqref{ineq:2-uniform convexity matrix} is reduced to
	\begin{equation}\label{ineq:key}
	\left.\dfrac{d^2}{dt^2}\right|_{t=0}\tr \left(f_p(A+tB)\right)\ge p(p-1)\|A\|_p^{p-2}\|B\|_p^2.
	\end{equation} 
	which is the main obstacle. Indeed, set $\varphi(t):=\|A+tB\|^2_p-(p-1)t^2\|B\|_p^2$. Then \eqref{ineq:2-uniform convexity matrix} is nothing but 
	\begin{equation*}
	\varphi(1)+\varphi(-1)\ge 2\varphi(0).
	\end{equation*}
	So it suffices to prove the convexity of $\varphi$, that is,
	\begin{equation}\label{ineq:pseudo-t}
	\dfrac{d^2}{dt^2}\|A+tB\|^2_p\ge 2(p-1)\|B\|_p^2.
	\end{equation}
	 By replacing $A$ with $A+tB$, it suffices to prove 
	\begin{equation}\label{ineq:pseudo-0}
	\left.\dfrac{d^2}{dt^2}\right|_{t=0}\|A+tB\|^2_p\ge 2(p-1)\|B\|_p^2.
	\end{equation}
	Put $\psi(t):=\tr\left(f_p(A+tB)\right)$. Then $\psi(0)=\tr|A|^p\in(0,\infty)$. Thus by Lemma \ref{lem:double operators}, $\|A+tB\|^2_p=\psi(t)^{2/p}$ is twice differentiable at $0$ and 
	\begin{align*}
	\left.\dfrac{d^2}{dt^2}\right|_{t=0}\|A+tB\|^2_p
	=&\left.\dfrac{d^2}{dt^2}\right|_{t=0}\left(\psi(t)^{\frac{2}{p}}\right)\\
	=&\frac{2}{p}\psi(0)^{\frac{2}{p}-1}\psi''(0)+\frac{2}{p}\left(\frac{2}{p}-1\right)\psi(0)^{\frac{2}{p}-2}\psi'(0)^2\\
	\ge &\frac{2}{p}\psi(0)^{\frac{2}{p}-1}\psi''(0)\\
	=&\frac{2}{p}\|A\|_p^{2-p}\psi''(0),
	\end{align*}
	where we have used the fact that $p\le 2$. This, together with \eqref{ineq:pseudo-0}, allows us to reduce the proof of \eqref{ineq:2-uniform convexity matrix} to \eqref{ineq:key}.
	
	Now we turn to the proof of \eqref{ineq:key}, which is different from the one in \cite{BCL94}. Let $A=\sum_{i=1}^{n}\lambda_i E^A_i$ be the spectral decomposition of $A$. From Lemma \ref{lem:double operators}, it follows that
	\begin{align*}
	\left.\dfrac{d^2}{dt^2}\right|_{t=0}\tr \left(f_p(A+tB)\right)
	=&2\sum_{i}\sum_{j}\sum_{k}f_p^{[2]}(\lambda_i,\lambda_j,\lambda_k)\tr\left(E^A_iBE^A_jBE^A_k\right)\\
	=&2\sum_{i}\sum_{j}f_p^{[2]}(\lambda_i,\lambda_j,\lambda_i)\tr\left(BE^A_jBE^A_i\right),
	\end{align*}
	where we have used the cyclicity of the trace and $E^A_{i}E^A_{k}=\delta_{ik}E^{A}_i$. Note that $\tr\left(BE^A_jBE^A_i\right)=\tr\left(BE^A_iBE^A_j\right)$, then 
	\begin{align*}
	\left.\dfrac{d^2}{dt^2}\right|_{t=0}\tr \left(f_p(A+tB)\right)
	=&\sum_{i}\sum_{j}\left(f_p^{[2]}(\lambda_i,\lambda_j,\lambda_i)
	+f_p^{[2]}(\lambda_j,\lambda_i,\lambda_j)\right)\tr\left(BE^A_jBE^A_i\right).
	\end{align*}
	
	We claim first that to prove \eqref{ineq:key} we may assume that $A$ is positive definite. The argument here is different from Ball-Carlen-Lieb's. In fact, the right hand side of \eqref{ineq:key} remains unchanged if we replace $A$ with $|A|$. Then it reduces to show that the left hand side of \eqref{ineq:key} is non-increasing when replacing $A$ with $|A|$. Clearly, $\tr\left(BE^A_iBE^A_j\right)\ge 0$ for all $i,j$. So it suffices to show that for all $i,j$
	\begin{equation*}
	f_p^{[2]}(\lambda_i,\lambda_j,\lambda_i)+f_p^{[2]}(\lambda_j,\lambda_i,\lambda_j)\ge f_p^{[2]}(|\lambda_i|,|\lambda_j|,|\lambda_i|)+f_p^{[2]}(|\lambda_j|,|\lambda_i|,|\lambda_j|).
	\end{equation*}
	We know by Lemma \ref{lem:properties of fp} that this is true and thus finish the proof of our claim. From now on, $A$ is assumed to be positive definite. 
	
	Let $\E$ be the conditional expectation from $M_{n\times n}(\mathbb{C})$ (suppose that $A,B$ are $n\times n$ matrices) onto the unital von Neumann subalgebra generated by $B$. In particular, it is a unital completely positive trace preserving map and $\E(B)=B$. Recall that
	\begin{align*}
	\left.\dfrac{d^2}{dt^2}\right|_{t=0}\tr \left(f_p(A+tB)\right)
	=&\sum_{i}\sum_{j}\left(f_p^{[2]}(\lambda_i,\lambda_j,\lambda_i)
	+f_p^{[2]}(\lambda_j,\lambda_i,\lambda_j)\right)\tr\left(BE^A_jBE^A_i\right)\\
	=&\sum_{i}\sum_{j}F_p(\lambda_i,\lambda_j)\tr\left(BE^A_jBE^A_i\right)\\
	=&\left\langle B,Q^{A,A}_{F_p}(B)\right\rangle.
	\end{align*}
	Here $F_p$ is a function on $(0,\infty)\times(0,\infty)$ which, by \eqref{eq:r r r} and \eqref{eq:r,s}, is of the form
	\begin{equation}\label{eq:F_p}
	F_p(r,s)=f_p^{[2]}(r,s,r)+f_p^{[2]}(s,r,s)
	=\begin{cases*}
	\frac{p(r^{p-1}-s^{p-1})}{r-s}& $r\ne s$\\
	p(p-1)r^{p-2}&$r=s$
	\end{cases*}.
	\end{equation} 
	So $F_{p}=pf_{\alpha}^{[1]}$, where $\alpha=p-1\in (0,1)$ and $f_{\alpha}=x^{\alpha},x>0$ is the function in Lemma \ref{lem:monotonicity}. Now we apply Lemma \ref{lem:monotonicity} (2) to obtain 
	\begin{equation}\label{ineq:monotonicity under conditional expectation}
	\left\langle B,Q^{A,A}_{F_p}(B)\right\rangle
	\ge \left\langle \E(B),Q^{\E(A),\E(A)}_{F_p}(\E(B))\right\rangle
	=\left\langle B,Q^{\E(A),\E(A)}_{F_p}(B)\right\rangle,
	\end{equation}
	which is nothing but 
	\begin{align}\label{ineq:2nd derivative}
	\left.\dfrac{d^2}{dt^2}\right|_{t=0}\tr \left(f_p(A+tB)\right)
	\ge &\left.\dfrac{d^2}{dt^2}\right|_{t=0}\tr \left(f_p(\E(A)+t B)\right).
	\end{align}
	Here, to prove \eqref{ineq:monotonicity under conditional expectation} one can also use the joint convexity (Lemma \ref{lem:monotonicity} (1)) of 
	\begin{equation}\label{eq:two variable function convex}
	(A,B)\mapsto \left\langle B,Q^{A,A}_{F_p}(B)\right\rangle,
	\end{equation}
	since the conditional expectation $\E$ is of the form
	\begin{equation*}
	\E(X)=\sum_{i=1}^{m}\alpha_iU_i^*XU_i,
	\end{equation*}
	where $m$ is some positive integer, $\alpha_i$'s are positive numbers summing up to 1, and $U_i$'s are unitaries that commute with $B$. 
	
	With \eqref{ineq:2nd derivative}, the remaining of the proof is a standard argument from \cite{BCL94}. Recall that $A$ is positive definite, so $\E(A)+tB$ is positive definite for small $|t|$ and for such $t$
	\begin{equation*}
	\tr \left(f_p(\E(A)+tB)\right)=\tr \left[\left(\E(A)+tB\right)^p\right].
	\end{equation*}
	Since $\E(A)$ commutes with $B$, we have
	\begin{equation*}
	\left.\dfrac{d^2}{dt^2}\right|_{t=0}\tr \left(f_p(\E(A)+tB)\right)
	=p(p-1)\tr \left(\E(A)^{p-2}B^2\right).
	\end{equation*}
	This, together with H\"older's inequality and the fact that $\|\E(A)\|_p\le \|A\|_p$ (since $\E$ is a conditional expectation), yields
	\begin{align*}
	\left.\dfrac{d^2}{dt^2}\right|_{t=0}\tr \left(f_p(\E(A)+tB)\right)
	\ge&p(p-1)\tr \left(\E(A)^{p-2}B^2\right)\\
	\ge& p(p-1)\|\E(A)\|_p^{p-2}\|B\|_p^2\\
	\ge&p(p-1)\|A\|_p^{p-2}\|B\|_p^2.
	\end{align*}
	This finishes the proof of \eqref{ineq:key} and thus \eqref{ineq:2-uniform convexity matrix}.
\end{proof}

\appendix

\section{Proof of Lemma \ref{lem:monotonicity} (2)}
\label{appendix}

Now let us give the proof of Lemma \ref{lem:monotonicity} (2).
\begin{proof}[Proof of Lemma \ref{lem:monotonicity} (2)]
	The desired result \eqref{ineq:monotonicity} is equivalent to
	\begin{equation}\label{ineq:f alpha}
	\beta^* Q_{f_{\alpha}^{[1]}}^{\beta(A),\beta(B)}\beta\le Q_{f_{\alpha}^{[1]}}^{A,B},
	\end{equation}
	where $\beta^*$ denotes the adjoint of $\beta$ with respect to $\langle\cdot,\cdot\rangle$. Recall that $\alpha-1\in (-1,0)$, thus we have the following integral representation
	\begin{equation}
	s^{\alpha-1}=\frac{\pi}{\sin(\alpha\pi)}\int_{0}^{\infty}\frac{t^{\alpha-1}}{t+s}dt,~~s>0.
	\end{equation}
	By the fundamental theorem of calculus and Fubini's theorem, we have
	\begin{equation}
	\begin{split}
	x^{\alpha}-y^{\alpha}
	&=\frac{1}{\alpha}\int_{y}^{x}s^{\alpha-1}ds\\
	&=\frac{\pi}{\alpha\sin(\alpha\pi)}\int_{y}^{x}\int_{0}^{\infty}\frac{t^{\alpha-1}}{t+s}dtds\\
	&=\frac{\pi}{\alpha\sin(\alpha\pi)}\int_{0}^{\infty}t^{\alpha-1}\left(\log(x+t)-\log(y+t)\right)dt.
	\end{split}
	\end{equation}
	Set $g_t(x,y):=\log(x+t)-\log(y+t)$. Then 
	\begin{equation*}
	f_{\alpha}^{[1]}(x,y)=\frac{\pi}{\alpha\sin(\alpha\pi)}\int_{0}^{\infty}t^{\alpha-1}g_{t}^{[1]}(x,y)dt,
	\end{equation*}
	and we may write $Q^{A,B}_{f_{\alpha}^{[1]}}$ as
	\begin{equation*}
	Q^{A,B}_{f_{\alpha}^{[1]}}=\frac{\pi}{\alpha\sin(\alpha\pi)}\int_{0}^{\infty}t^{\alpha-1}Q^{A,B}_{g_t^{[1]}}dt.
	\end{equation*}
	Hence to prove \eqref{ineq:f alpha} it suffices to show that for all $t>0$ and positive definite $A,B$
	\begin{equation}\label{ineq:g t}
	\beta^* Q_{g_t^{[1]}}^{\beta(A),\beta(B)}\beta\le Q_{g_t^{[1]}}^{A,B}.
	\end{equation}
	For any $t>0$, note that for $F_t(x,y):=F(x+t,y+t)$ there holds
	\begin{equation*}
	Q_{F_t}^{A,B}=Q_F^{A+t I, B+tI},
	\end{equation*}
	where $I$ is the identity matrix. Since $\beta$ is linear and unital, we have 
	\begin{equation*}
	Q_{F_t}^{\beta(A),\beta(B)}=Q_F^{\beta(A)+t I, \beta(B)+tI}=Q_F^{\beta(A+t I), \beta(B+tI)}.
	\end{equation*}
	Therefore to prove \eqref{ineq:g t}, it suffices to show  that for all positive definite $A,B$
	\begin{equation}\label{ineq:g}
	\beta^* Q_{g^{[1]}}^{\beta(A),\beta(B)}\beta\le Q_{g^{[1]}}^{A,B},
	\end{equation}
	where $g(x,y)=\log x-\log y$. This is a known result due to Hiai and Petz \cite{HP12quasi}. In fact, put
	\begin{equation*}
	J_f^{A,B}:=f(L_A R_{B}^{-1})R_B,
	\end{equation*}
	where $L_A$ and $R_B$ are the left multiplication and right multiplication operators respectively. Note that for $F(x,y):=f(xy^{-1})y$ we have
	\begin{equation}\label{eq:identification}
	Q_{F}^{A,B}=J_{f}^{A,B}.
	\end{equation}
	 Hiai and Petz proved that \cite[Theorem 5]{HP12quasi} if $h:(0,\infty)\to (0,\infty)$ is operator monotone, then 
	\begin{equation}\label{ineq: h inverse}
	\beta^* \left(J_{h}^{\beta(A),\beta(B)}\right)^{-1}\beta\le \left(J_{h}^{A,B}\right)^{-1}.
	\end{equation}
	Since $Q_{F^{-1}}^{A,B}=(Q_{F}^{A,B})^{-1}$ and
	\begin{equation*}
	(g^{[1]})^{-1}(x,y)
	=\begin{cases*}
	\frac{x-y}{\log x-\log y}=\frac{\frac{x}{y}-1}{\log\left(\frac{x}{y}\right)}\cdot y&$x\ne y$\\
	x&$x=y$
	\end{cases*},
	\end{equation*}
	we have by \eqref{eq:identification} that
	\begin{equation*}
	Q_{g^{[1]}}^{A,B}=\left(Q_{(g^{[1]})^{-1}}^{A,B}\right)^{-1}
	=\left(J_{h}^{A,B}\right)^{-1},
	\end{equation*}
	with
	\begin{equation*}
	h(x):=\begin{cases*}
	\frac{x-1}{\log x}&$x\ne 1$\\
	1&$x=1$
	\end{cases*}.
	\end{equation*}
	To finish the proof, it remains to show that $h$ is operator monotone, which can be seen from $h(x)=\int_{0}^{1}x^t dt$.
\end{proof}

\subsection*{Acknowledgements.} The research was supported by the European Union's Horizon 2020 research and innovation programme under the Marie Sk\l odowska-Curie grant agreement No. 754411. The author would like to thank Haojian Li for helpful discussions and a careful reading of the draft.

\newcommand{\etalchar}[1]{$^{#1}$}

\end{document}